\documentclass[final,leqno,onefignum,onetabnum]{siamltex1213}

\usepackage[dvips]{color}
\usepackage{graphicx}

\usepackage{amsmath} 
\usepackage{amssymb}

\usepackage{enumitem}

\title{
Construction of Lyapunov functions for interconnected parabolic systems: an iISS approach
\thanks{This work was supported by Japanese Society for Promotion of Science (JSPS).}
} 

\author{Andrii Mironchenko \thanks{Andrii Mironchenko is with Department of Systems Design and Informatics, Kyushu Institute of Technology, 680-4 Kawazu, Iizuka, Fukuoka 820-8502, Japan 
(\email{andrii.mironchenko@mathematik.uni-wuerzburg.de}). Corresponding author.}
\and Hiroshi Ito
\thanks{Hiroshi Ito is with Department of Systems Design and Informatics, Kyushu Institute of Technology, 680-4 Kawazu, Iizuka, Fukuoka 820-8502, Japan 
(\email{hiroshi@ces.kyutech.ac.jp}).
}
}

\newtheorem{Satz}{Theorem} 
\newtheorem{Def}{Definition} 

\newtheorem{remark}[Satz]{Remark}

\newtheorem{Ass}{Assumption}

\newcommand \N   {\mathbb{N}}
\newcommand \R   {\mathbb{R}}

\newcommand \K   {\mathcal{K}}
\newcommand \Kinf{\mathcal{K_\infty}}
\newcommand \PD  {\mathcal{P}}
\newcommand \KL  {\mathcal{KL}}
\newcommand \LL  {\mathcal{L}}

\begin{document}
\maketitle

\begin{abstract}
This paper is devoted to two issues. One is to provide 
Lyapunov-based tools to establish integral input-to-state stability (iISS) and input-to-state stability (ISS) for some classes of nonlinear parabolic equations. 
The other is to provide a stability criterion for interconnections 
of iISS parabolic systems. 
The results addressing the former problem allow us to overcome obstacles 
arising in tackling the latter one. 
The results for the latter problem are a small-gain condition and a formula of 
Lyapunov functions which can be constructed for 
interconnections whenever the small-gain condition holds. 
It is demonstrated that for interconnections of partial differential equations, the choice of a right state and input spaces is crucial, in particular for iISS subsystems which are not ISS. 
As illustrative examples, 
stability of two highly nonlinear reaction-diffusion systems is established 
by the the proposed small-gain criterion.
\end{abstract}



\pagestyle{myheadings}
\thispagestyle{plain}
\markboth{iISS OF INTERCONNECTED PARABOLIC SYSTEMS}{iISS OF INTERCONNECTED PARABOLIC SYSTEMS}

\section{Introduction}


During more than two decades, input-to-state stability (ISS) has been 
used widely in the study of robust stabilizability \cite{FrK08}, 
detectability \cite{SoW97,KSW01}
and other branches of nonlinear control theory \cite{KoA01,Son08}. 
ISS unified into one framework two different types of stable behavior: 
asymptotic stability and input-output stability \cite{Son08}, and allowed us to use the small-gain argument to study ISS of interconnected systems \cite{JTP94,JMW96}, which is sometimes referred to as the ISS small-gain theorem. 
In spite of these advantages, for many practical systems ISS is still far too restrictive.
This is because ISS excludes systems whose state stays bounded as long as 
the magnitude of the applied inputs remains below a specific threshold, but becomes unbounded when the input magnitude exceeds the threshold. Such behavior is frequently caused by saturation and limitations in actuation and processing rate. 
The idea of integral input-to-state stability (iISS) is 
to capture those nonlinearities \cite{Son98,ASW00}. 
Serious obstacles were encountered in addressing interconnections of iISS systems \cite{ITOnolcos13}. In contrast to 
ISS subsystems, iISS subsystems which 
are not ISS usher the issue of incompatibility of spaces 
in time domain for trajectory-based approaches, as well as 
insufficiency of max-type Lyapunov functions popular in 
ISS Lyapunov-based approaches (e.g. \cite{JMW96, DRW10}). 
Breakthroughs made in \cite{ITOTAC06,ItJ09,ANGSGiISS,KaJ12} 
allowed us to use small-gain criteria 
as in the ISS small-gain theorem 
in spite of the inevitable and considerable difference between 
their proofs and Lyapunov constructions.  

In contrast to the extensive literature on ISS and iISS of ordinary 
differential equations, for partial differential equations (PDEs) these theories are making 
their first steps. 
In \cite{JLR08}, \cite{DaM13}, \cite{DaM13b}, \cite{Log13}, ISS of infinite-dimensional systems
\begin{equation}
\label{InfiniteDim}
\dot{x}(t)=Ax(t)+f(x(t),u(t)), \ x(t) \in X, u(t) \in U 
\end{equation}
has been studied via methods of semigroup theory \cite{JaZ12}, \cite{CaH98}. 
Here the state space $X$ and the space of input values $U$ are Banach spaces, $A:D(A) \to X$ is the generator of a $C_0$-semigroup over $X$ and $f:X \times U \to X$ is Lipschitz with respect to  the first argument. Many classes of evolution PDEs, such as parabolic and hyperbolic equations are of this kind \cite{Hen81}, \cite{CaH98}.
As in the case of finite-dimensional systems \cite{Son08}, 
the notion of an ISS Lyapunov function can be defined 
for \eqref{InfiniteDim} so that the 
existence of an ISS Lyapunov function is sufficient for ISS of 
\eqref{InfiniteDim} (see \cite{DaM13}). 
This motivated the results in \cite{DaM13} 
on constructions of ISS Lyapunov functions for a class of parabolic 
systems belonging to \eqref{InfiniteDim}. 
More direct approach to the  construction of Lyapunov functions for some classes of 
nonlinear parabolic and linear time-varying hyperbolic systems has been proposed in \cite{MaP11, PrM12}. 
In \cite{JLR08} and \cite{Log13}, systems \eqref{InfiniteDim} with a linear function $f$ have been investigated via frequency-domain methods. %

To the best of the authors' knowledge, with some exceptions of time-delay systems, the study of Lyapunov functions for iISS of 
infinite-dimensional systems has begun in \cite{MiI14b}, 
where iISS of bilinear distributed parameter systems was 
investigated.
It was shown that bilinear systems in the form of \eqref{InfiniteDim} which are uniformly globally 
asymptotically stable without inputs are always iISS. 
The second result in \cite{MiI14b} is an extension to bilinear systems over Hilbert spaces 
of a method for construction of iISS Lyapunov functions for bilinear ODE systems, introduced by Sontag \cite{Son98}. In this paper we use similar method in Section~\ref{sec:parablic} to construct iISS Lyapunov functions for some classes of nonlinear parabolic systems.

In \cite{DaM13,DaM13b}, ISS of large scale systems whose subsystems 
are in the form of \eqref{InfiniteDim} has been studied and the ISS small gain theorem, already available for finite-dimensional systems (see \cite{JMW96,DRW10})
has been extended to the infinite-dimensional systems. 
 However, the method does not accommodate iISS subsystems 
which are not ISS. 

This paper studies stability of interconnections of two parabolic systems, each of which is of the form
\begin{equation}
\label{Nonlinear_Parabolic_General}
\frac{\partial x}{\partial t} = c \frac{\partial^2 x}{\partial l^2} + f(x(l,t),\tfrac{\partial x}{\partial l}(l,t),u(l,t)),\quad  \forall  t>0,
\end{equation}
where $l \in (0,L)$, $x(l,t) \in \R$. This class of systems \eqref{Nonlinear_Parabolic_General} allows more general functions $f$ than the class considered in \cite{DaM13,MiI14b}, and possesses systems, which are not ISS. 
The primary goal of this paper is to accomplish 
an iISS small gain theorem \cite{ItJ09,ITOnolcos13}, 
originally proved for finite-dimensional systems, 
in the infinite-dimensional setting. 
In contrast to the small-gain theorem from \cite{DaM13}, we require ISS property only from one subsystem and not from both of them; the other subsystem may be only iISS. 

Interestingly, 
this extension is much more involved than it seems on the first glance. 
When working with PDEs which are not ISS, an obstacle is not only in a higher complexity in dealing with 
Lyapunov functions in infinite dimensions, 
but also in the necessity to choose the state space in a right way. In particular, it is quite hard to find an iISS parabolic system whose state and input spaces are both $L_p$-spaces, while we do not encounter such difficulties considering ISS systems. To address this issue, this paper reexamines tools 
developed in \cite{DaM13,MiI14b} for constructing iISS and ISS 
Lyapunov functions for some classes of nonlinear parabolic systems, and actively exploits Sobolev spaces as state spaces.
For interconnections of PDE systems additional difficulties arise since we need not only choose right spaces for every subsystem, but also match them with the state and input spaces for another subsystems. Last but not least, 
incompatibility of spaces in the time domain, which is crucial for interconnections of ODE systems, is as important for PDE systems. These issues make an investigation of interconnections of iISS infinite-dimensional systems a challenging problem, which we solve here for some classes of parabolic systems in one-dimensional spatial domain. 


The rest of this paper is organized as follows. 
Having defined the stability notions in 
Section~\ref{Problem_Formulation}, a construction of an ISS 
Lyapunov function for nonlinear parabolic systems 
\eqref{Nonlinear_Parabolic_General} 
with Sobolev state space is proposed  in Section~\ref{sec:ISS_Parabolic}, 
in which the role of the input space is also highlighted. 
In Section~\ref{sec:parablic} we consider another class of nonlinear parabolic 
systems, and provide a construction of an iISS Lyapunov function with $L_p$ state space. 
In Section~\ref{sec:parablicsobolev}, a construction of an 
iISS Lyapunov function for \eqref{Nonlinear_Parabolic_General} 
is developed with Sobolev state space in order to 
go beyond systems dealt in Sections~\ref{sec:ISS_Parabolic}, \ref{sec:parablic}. 
Next in Section~\ref{GekoppelteISS_Systeme} we state an 
iISS small-gain theorem which is a sufficient condition for iISS 
of the interconnection of two iISS distributed parameter systems, and an iISS Lyapunov function is constructed explicitly for the interconnection. 
As it was proved to be necessary for finite-dimensional systems \cite{ItJ09}, 
one subsystem is required to be ISS if the other subsystem is iISS and does not admit an ISS Lyapunov function. Finally, in Section~\ref{sec:Example} we use all of the obtained results to prove stability of two types of highly nonlinear parabolic systems. The examples illustrate how fruitful the use of Sobolev spaces is in dealing with interconnections involving iISS systems. Finally, conclusions are drawn in Section~\ref{sec:conc}.

\subsection*{Notation}
\label{sec:Notation}


%
We define $\R_+:=[0,\infty)$, and 
the symbol $\N$ denotes the set of natural numbers. 
By $C(\R_+,Y)$ we denote the space of continuous functions from $\R_+$ to $Y$, equipped with the standard $\sup$-norm.
We exploit throughout the paper the following function spaces: 
\begin{itemize}[leftmargin=8mm]
  \item $C_0^k(0,L)$ is a space of $k$ times continuously differentiable functions 
	$f:(0,L) \to \R$ with a support, compact in $(0,L)$.
	\item $L_p(0,L)$, $p \geq 1$ is a space of $p$-th power integrable functions $f:(0,L) \to \R$ with the norm $\|f\|_{L_p(0,L)}=\left( \int_0^L{|f(l)|^p dl} \right)^{\frac{1}{p}}$.
	\item $W^{k,p}(0,L)$ is a Sobolev space of functions $f \in L_p(0,L)$, which have weak derivatives of order $\leq k$, all of which belong to $L_p(0,L)$.
	Norm in $W^{k,p}(0,L)$ is defined by $\|f\|_{W^{k,p}(0,L)}=\Big( \int_0^L{\sum_{1 \leq s \leq k}\left|\frac{\partial^s f}{\partial l^s}(l)\right|^p dl} \Big)^{\frac{1}{p}}$.
	\item $W^{k,p}_0(0,L)$ is a closure of $C_0^k(0,L)$ in the norm of $W^{k,p}(0,L)$. We endow $W^{k,p}_0(0,L)$ with a norm $\|f\|_{W^{k,p}_0(0,L)}=\Big( \int_0^L{\left|\frac{\partial^k f}{\partial l^k}(l)\right|^p dl} \Big)^{\frac{1}{p}}$,
equivalent to the norm  $\|~\cdot~\|_{W^{k,p}(0,L)}$ on $W^{k,p}_0(0,L)$, see, \cite[p.8]{Hen81}.
	\item $H^k(0,L)=W^{k,2}(0,L)$,  $H^k_0(0,L)=W^{k,2}_0(0,L)$.
\end{itemize}


To define and analyze stability properties we use so-called comparison functions
\begin{equation*}
\begin{array}{ll}
{\PD} &:= \left\{\gamma:\R_+\rightarrow\R_+\left|\ \gamma\mbox{ is continuous, }  \gamma(0)=0 \mbox{ and } \gamma(r)>0 \mbox{ for } r>0 \right. \right\} \\
{\K} &:= \left\{\gamma \in \PD \left|\ \gamma \mbox{ is strictly increasing}  \right. \right\}\\
{\K_{\infty}}&:=\left\{\gamma\in\K\left|\ \gamma\mbox{ is unbounded}\right.\right\}\\
{\LL}&:=\{\gamma:\R_+\rightarrow\R_+\left|\ \gamma\mbox{ is continuous and strictly decreasing with } \lim\limits_{t\rightarrow\infty}\gamma(t)=0\right. \}\\
{\KL} &:= \left\{\beta:\R_+\times\R_+\rightarrow\R_+\left|\ \beta \mbox{ is continuous, } \beta(\cdot,t)\in{\K},\ \forall t \geq 0, \  \beta(r,\cdot)\in {\LL},\ \forall r > 0\right.\right\} \\
\end{array}
\end{equation*}


\begin{table}[t]
\caption{Some useful spaces for 
interconnections allowing for {\rm i}ISS subsystems.}
\label{table:iISS}
\begin{center}
\vspace{-2ex}
\begin{tabular}{c}
(a) Choice \#1
\\[.2ex]
\begin{tabular}{lll}
\hline
& State values $X_i$ & Input values $U_i$ \\
\hline
iISS subsystem $(i=1)$ & $L_2(0,L)$   & $H_0^1(0,L)$ 
\\ 
ISS subsystem  $(i=2)$ & $H_0^1(0,L)$ & $L_2(0,L)$ 
\\
\hline
\end{tabular}
\\[4ex]
(b) Choice \#2
\\[.2ex]
\begin{tabular}{lll}
\hline
& State values $X_i$ & Input values $U_i$ \\
\hline
iISS subsystem $(i=1)$ & $H_0^1(0,L)$ & $H_0^1(0,L)$ 
\\ 
ISS subsystem  $(i=2)$ & $H_0^1(0,L)$ & $H_0^1(0,L)$ 
\\
\hline
\end{tabular}
\end{tabular}
\end{center}
\end{table}

\begin{table}[t]
\caption{A typical choice of spaces for interconnections of ISS subsystems.}
\label{table:ISS}
\begin{center}
\vspace{-2ex}
\begin{tabular}{lll}
\hline
& State values $X_i$ & Input values $U_i$ \\ \hline
ISS subsystem $(i=1)$ & $L_p(0,L)$ & $L_q(0,L)$ \\ 
ISS subsystem $(i=2)$ & $L_q(0,L)$ & $L_p(0,L)$ \\
\hline
\end{tabular}
\end{center}
\end{table}


\section{Problem formulation}
\label{Problem_Formulation}

Consider the system \eqref{InfiniteDim} and assume throughout the paper that $X$ and $U$ are Banach spaces and $f(0,0)=0$, i.e., 
$x=0\in X$ is an equilibrium point of \eqref{InfiniteDim}. Let also $T$ be a semigroup generated by $A$ from \eqref{InfiniteDim}.

Under (weak) solutions of \eqref{InfiniteDim} we understand solutions of the integral equation
\begin{align}
\label{InfiniteDim_Integral_Form}
x(t)=T(t)x(0) + \int_0^t \hspace{-1ex}T(t-s)f(x(s),u(s))ds,
\end{align}
where $\forall t\in[0,\tau]$, belonging to $C([0,\tau],X)$ for all $\tau>0$.

\begin{Def}
We call $f:X \times U \to X$ Lipschitz continuous on bounded subsets of $X$, uniformly with respect to  the second argument if $\forall C>0 \; \exists K(C)>0$, such that $\forall x,y: \|x\|_X \leq C,\ \|y\|_X \leq C$, $\forall v \in U$, 
\begin{eqnarray}
\|f(y,v)-f(x,v)\|_X \leq K(C) \|y-x\|_X.
\label{eq:Lipschitz}
\end{eqnarray}
\end{Def}

We will use the following assumption concerning nonlinearity $f$ throughout the paper
\begin{Ass}
\label{Assumption1}
$f:X \times U \to X$ is Lipschitz continuous on bounded subsets of $X$, uniformly with respect to  the second argument and $f(x,\cdot)$ is continuous for all $x \in X$.
\end{Ass}

Assumption~\ref{Assumption1} ensures that the weak solution of \eqref{InfiniteDim} exists and is unique, according to a variation of a classical existence and uniqueness theorem \cite[Proposition 4.3.3]{CaH98}.
Let $\phi(t,\phi_0,u)$ denote the state of a system 
\eqref{InfiniteDim}, i.e. the solution to \eqref{InfiniteDim}, 
at moment $t\in\R_+$ associated with 
an initial condition $\phi_0\in X$ at $t=0$, and input $u\in U_c$, 
where $U_c$ is a linear normed space of admissible 
functions mapping $\R_+$ into $U$, equipped 
with a norm $\|\cdot\|_{U_c}$. 

Next we introduce stability properties for the system \eqref{InfiniteDim}.
\begin{Def}
System \eqref{InfiniteDim} is {\it globally asymptotically
stable at zero uniformly with respect to state} (0-UGASs), if $\exists \beta \in \KL$, such that $\forall \phi_0 \in X$, $\forall t\geq 0$ it holds
\begin{equation}
\label{UniStabAbschaetzung}
\left\| \phi(t,\phi_0,0) \right\|_{X} \leq  \beta(\left\| \phi_0 \right\|_{X},t) .
\end{equation}
%
\end{Def}

%


To study stability properties of \eqref{InfiniteDim} with respect to external inputs, we use the notion of input-to-state stability \cite{DaM13}:
\begin{Def}
\label{Def:ISS}
System \eqref{InfiniteDim} is called {\it input-to-state stable
(ISS) with respect to  space of inputs $U_c$}, if there exist $\beta \in \KL$ and $\gamma \in \K$ 
such that the inequality
\begin {equation}
\label{iss_sum}
\begin {array} {lll}
\| \phi(t,\phi_0,u) \|_{X} \leq \beta(\| \phi_0 \|_{X},t) + \gamma( \|u\|_{U_c})
\end {array}
\end{equation}
holds $\forall \phi_0 \in X$, $\forall u\in U_c$ and $\forall t\geq 0$.

We emphasize that the above definition does not yet exactly 
correspond to ISS of finite dimensional systems \cite{Son08} since 
Definition \ref{Def:ISS} allows the flexibility of the choice of $U_c$. 
A \textit{system \eqref{InfiniteDim} is called ISS}, without expressing the 
normed space of inputs explicitly, if it is ISS with respect to  $U_c=C(\R_+,U)$ 
endowed with a usual supremum norm. 
This terminology follows that of ISS for finite dimensional systems. 

\end{Def}


If the system is not ISS, it may still have some sort of robustness. Thus we introduce another stability property 
\begin{Def}
\label{Def:iISS}
System \eqref{InfiniteDim} is called {\it integral input-to-state stable (iISS)} if there exist $\alpha \in \Kinf$, $\mu \in \K$ and $\beta \in \KL$ 
such that the inequality 
\begin{equation}
\label{iISS_Estimate}
\alpha(\|\phi(t,\phi_0,u)\|_X) \leq \beta(\|\phi_0\|_X,t) + \int_0^t \mu(\|u(s)\|_U)ds
\end{equation}
holds $\forall \phi_0 \in X$, $\forall u\in U_c=C(\R_+,U)$ and $\forall t\geq 0$.
\end{Def}

The following defines a useful notion for studying iISS. 

\begin{Def}\label{Def:iISSV}
A continuous function $V:X \to \R_+$ is called an \textit{iISS Lyapunov function},  if there exist
$\psi_1,\psi_2 \in \Kinf$, $\alpha \in \PD$ and $\sigma \in \K$ 
such that 
\begin{equation}
\label{LyapFunk_1Eig}
\psi_1(\|x\|_X) \leq V(x) \leq \psi_2(\|x\|_X), \quad \forall x \in X
\end{equation}
and system \eqref{InfiniteDim} satisfies 
\begin{equation}
\label{DissipationIneq}
\dot{V}_u(x) \leq -\alpha(\|x\|_X) + \sigma(\|u(0)\|_U)
\end{equation}
for all $x \in X$ and $u\in U_c$, 
where the Lie derivative of $V$ corresponding to the input $u$ is defined by
\begin{equation}
\label{LyapAbleitung}
\dot{V}_u(x)=\mathop{\overline{\lim}} \limits_{t \rightarrow +0} {\frac{1}{t}(V(\phi(t,x,u))-V(x)) }.
\end{equation}
Furthermore, if 
\begin{equation}
\label{eq:ISSalphsig}
\liminf_{\tau\to\infty}\alpha(\tau)=\infty 
\,\mbox{ or }\, 
\liminf_{\tau\to\infty}\alpha(\tau)\geq
\lim_{\tau\to\infty}\sigma(\tau)
\end{equation}
holds, system $V$ is called an \textit{ISS Lyapunov function}. 
\end{Def}

\begin{proposition}[Proposition 1, \cite{MiI14b}]
\label{PropSufiISS}
 If there exist an iISS (resp. ISS) Lyapunov function for \eqref{InfiniteDim}, then \eqref{InfiniteDim} is iISS (resp. ISS).
\end{proposition}

As a rule a construction of a Lyapunov function is the only realistic way to prove ISS/iISS of control systems.
This makes the construction of an ISS/iISS Lyapunov functions a fundamental problem in stability theory. 
In the next sections we propose a method for constructing iISS and ISS Lyapunov functions for certain equations \eqref{Nonlinear_Parabolic_General}. Then we show how to construct the Lyapunov functions for systems of PDEs from the information about Lyapunov functions of subsystems by means of an small-gain approach.

In this paper, for simplicity we  write $\dot{V}$ instead of 
$\dot{V}_u(x)$ when solutions along which the derivative 
is taken are clear from the context. 

\begin{remark}
For finite-dimensional systems iISS notion is strictly weaker than ISS, in sense that all ISS systems are iISS \cite{ASW00}. 
In addition, there are iISS systems which are not 
ISS \cite{Son98}. This strict inclusive relationship has not yet 
been proved completely for PDEs of the form 
\eqref{Nonlinear_Parabolic_General}. However, in terms of Lyapunov functions defined above, ISS Lyapunov functions establishing ISS are 
always iISS Lyapunov functions establishing iISS. Furthermore, iISS PDEs of the form  \eqref{Nonlinear_Parabolic_General} are not necessarily ISS, see \cite{MiI14b}. 
\end{remark}

\section{ISS Lyapunov functions for a class of nonlinear parabolic systems: Sobolev state space}
\label{sec:ISS_Parabolic}

The purpose of this section is to develop a Lyapunov-type 
characterization of ISS for PDEs in 
\eqref{Nonlinear_Parabolic_General}.
There is a number of papers, where such characterizations for parabolic systems whose state space is an $L_p$ space have been provided \cite{MaP11,DaM13}. However, as we will see in Section~\ref{sec:Example} the iISS systems in many cases cannot have the $L_p$ space both as an input and state space. Since our final goal is to consider interconnections of iISS and ISS systems, we need to have the constructions of ISS Lyapunov functions with Sobolev state spaces. This section provides one of such constructions.

Consider a system 
\begin{equation}
\label{Nonlinear_Parabolic_W_12q}
\frac{\partial x}{\partial t} = c \frac{\partial^2 x}{\partial l^2} + f\big(x(l,t),\tfrac{\partial x}{\partial l}(l,t)\big)+ u(l,t),\quad  \forall  t>0
\end{equation}
defined on the spatial domain $(0,L)$ with the Dirichlet boundary conditions
\begin{eqnarray}
x(0,t) = x(L,t) =0, \quad \forall t \geq 0 .
\label{eq:BoundaryConditions_W_12q}
\end{eqnarray}
The next theorem gives a sufficient condition for ISS 
of \eqref{Nonlinear_Parabolic_W_12q} with respect to the state space $X=W^{1,2q}_0(0,L)$, $q \in \N$ and two types of spaces $U$ of 
input values by construction of a Lyapunov function.

\begin{Satz}\label{theorem:parabolic_ISS_W_12q_norm}
Suppose 
\begin{eqnarray}
\int_0^L \Big( \frac{\partial x}{\partial l} \Big)^{2q-2} \frac{\partial^2 x}{\partial l^2}f\left(x,\tfrac{\partial x}{\partial l}\right) dl \geq 
\int_0^L \eta\left(\left(\tfrac{\partial x}{\partial l}\right)^{2q}\right) dl
\label{eq:Assumption_on_f}
\end{eqnarray}
holds all $x \in X$ with 
some convex continuous function $\eta: \R_+\to \R$ 
and some $\epsilon\in\R_+$ such that 
\begin{eqnarray}
\hat{\alpha}(s):=
\frac{\pi^2}{q^2L^2}(c-\epsilon)s+L\eta\left(\frac{s}{L}\right)
\geq 0, \ \forall s\in\R_+ . 
\label{eq:Assumption_on_a}
\end{eqnarray}
Then the following statements hold: 
\begin{enumerate}
\item\label{item:ISSLFitem1} 
If $\epsilon>0$, then the function 
\begin{align}
V(x)=\int_0^L \Big( \frac{\partial x}{\partial l} \Big)^{2q} dl = \|x\|_{W^{1,2q}_0(0,L)}^{2q} 
\label{eq:LF_W_12q_norm}
\end{align}
is an ISS Lyapunov function of 
\eqref{Nonlinear_Parabolic_W_12q}-\eqref{eq:BoundaryConditions_W_12q} with respect to the space $U=L_{2q}(0,L)$ of input values 
and $U=W_0^{1,2q}(0,L)\cap W^{2,2q}(0,L)$ as well. 
\item\label{item:ISSLFitem2} 
If there exists $g \in \Kinf$ so that
\begin{align}
Lsg(s)=\hat{\alpha}
\left(Ls^{\tfrac{2q}{2q-1}}\right)
, \quad \forall s\in\R_+
\label{eq:Assumption_on_alpha}
\end{align}
holds, then 
the function $V$ given in \eqref{eq:LF_W_12q_norm} 
is an ISS Lyapunov function of 
\eqref{Nonlinear_Parabolic_W_12q}-\eqref{eq:BoundaryConditions_W_12q} with respect to the space of input values $U=U_g$, consisting of $u \in L_1(0,L)$: $u(0)=u(L)=0$, $\frac{\partial u}{\partial l}$ exists and 
$\int_0^L \left| \frac{\partial u}{\partial l}\right|
g^{-1}\left(\left| \frac{\partial u}{\partial l}\right|\right) dl$ 
is finite.
\end{enumerate}
\end{Satz}
\begin{proof}
Along the solution of 
\eqref{Nonlinear_Parabolic_W_12q}-\eqref{eq:BoundaryConditions_W_12q}, 
the function $V$ given as in \eqref{eq:LF_W_12q_norm} satisfies 
\begin{align*}
\dot{V} = & \; 2q \int_0^L \Big( \frac{\partial x}{\partial l} \Big)^{2q-1} \frac{\partial^2 x}{\partial l \partial t } dl \\
					 = & -2q (2q-1) \int_0^L \frac{\partial x}{\partial t} \Big( \frac{\partial x}{\partial l} \Big)^{2q-2} \frac{\partial^2 x}{\partial l^2} dl +2q \Big( \frac{\partial x}{\partial l} \Big)^{2q-1} \frac{\partial x}{\partial t} \Big|_{l=0}^L.
\end{align*}
Due to \eqref{eq:BoundaryConditions_W_12q} we have $\frac{\partial x}{\partial t} \big|_{l=0}^L=0$ and consequently 					
\begin{align*}
\dot{V} = -2q (2q-1) \int_0^L \Big( \frac{\partial x}{\partial l} \Big)^{2q-2} \frac{\partial^2 x}{\partial l^2} 
				\cdot \Big(c \frac{\partial^2 x}{\partial l^2} + f(x(l,t),\tfrac{\partial x}{\partial l}(l,t))+ u(l,t) \Big) dl.				
\end{align*}
Next we utilize \eqref{eq:Assumption_on_f} to obtain
\begin{align}
\tfrac{1}{2q (2q-1)} \dot{V} 
\leq& -c \int_0^L \Big( \frac{\partial x}{\partial l} \Big)^{2q-2} \Big(\frac{\partial^2 x}{\partial l^2}\Big)^2 dl 
-\int_0^L \eta\left(\left(\tfrac{\partial x}{\partial l}\right)^{2q}\right) dl
\nonumber\\
				    & \ 
- \int_0^L \Big( \frac{\partial x}{\partial l} \Big)^{2q-2} \frac{\partial^2 x}{\partial l^2} u  dl.
\label{BasicEstimate_For_Other_U}
\end{align}
Using Young's inequality $\frac{\partial^2 x}{\partial l^2} u \leq \frac{\omega}{2} \Big(\frac{\partial^2 x}{\partial l^2}\Big)^2 + \tfrac{1}{2\omega}u^2$, which holds for any $\omega>0$, we get
\begin{align}
\tfrac{1}{2q (2q-1)} \dot{V} \leq& \left(\frac{\omega}{2}-c\right) \int_0^L \Big( \frac{\partial x}{\partial l} \Big)^{2q-2} \Big(\frac{\partial^2 x}{\partial l^2}\Big)^2 dl 
-\int_0^L \eta\left(\left(\tfrac{\partial x}{\partial l}\right)^{2q}\right) dl
\nonumber\\
				    & \ 
+ \frac{1}{2\omega}\int_0^L \Big( \frac{\partial x}{\partial l} \Big)^{2q-2} u^2 dl.
\label{eq:W_12q_Temporary_Estimate}
\end{align}
It is easy to see that
\[
\int_0^L \Big( \frac{\partial x}{\partial l} \Big)^{2q-2} \Big(\frac{\partial^2 x}{\partial l^2}\Big)^2 dl
= \frac{1}{q^2} \int_0^L \Big( \frac{\partial}{\partial l}    \Big(\Big( \frac{\partial x}{\partial l} \Big)^{q} \Big)\Big)^2 dl.
\]
Due to Friedrichs' inequality  \eqref{ineq:Friedrichs} we proceed to
\begin{equation}
\int_0^L \Big( \frac{\partial x}{\partial l} \Big)^{2q-2} \Big(\frac{\partial^2 x}{\partial l^2}\Big)^2 dl
\geq \frac{\pi^2}{q^2L^2} \int_0^L \Big( \frac{\partial x}{\partial l} \Big)^{2q}dl  = \frac{\pi^2}{q^2L^2}V(x).
\label{eq:LinTermEstimate}
\end{equation}
Define 
\begin{eqnarray}
\xi(s):=\frac{1}{L}\hat{\alpha}(Ls) = \frac{\pi^2}{q^2L^2}(c-\epsilon)s+\eta\left(s\right)
, \quad  \forall s\in\R_+ . 
\label{eq:Assumption_xi}
\end{eqnarray}
The convexity of $\eta$ implies the convexity of $\xi$. 
Due to the definition \eqref{eq:LF_W_12q_norm} of $V$, 
Jensen's inequality \eqref{ineq:Jensen} yields 
\begin{eqnarray}
\int_0^L\xi\left(\left(\tfrac{\partial x}{\partial l}\right)^{2q}\right) dl 
\geq \hat{\alpha}(V(x)) , \quad \forall x\in\R . 
\label{eq:xihatal}
\end{eqnarray}
Now we continue the estimates of $\dot{V}$ in the 
cases of $q=1$ and $q>1$ separately.

In the case of $q=1$, inequality \eqref{eq:W_12q_Temporary_Estimate} implies
\begin{align}
\dot{V} &\leq 
2\Big(\frac{\omega}{2}-c\Big)\frac{\pi^2}{L^2}V(x) 
-2\!\int_0^L \hspace{-2ex}\eta\left(\left(\tfrac{\partial x}{\partial l}\right)^{\!2}\right)\!dl
+ \frac{1}{\omega} \|u\|^{2}_{L_{2}(0,L)}
\nonumber \\
& = 
2\Big(\frac{\omega}{2}-c\Big)\frac{\pi^2}{L^2}V(x) 
-2\!\int_0^L \hspace{-2ex}\xi\left(\left(\tfrac{\partial x}{\partial l}\right)^{\!2}\right)\!dl
\nonumber \\
& \qquad 
+2\left(c-\epsilon\right)\frac{\pi^2}{L^2} \!\int_0^L \hspace{-1ex}\left(\frac{\partial x}{\partial l}\right)^{\!2} dl
+ \frac{1}{\omega} \|u\|^{2}_{L_{2}(0,L)}
\nonumber \\
& \leq 
2\Big(\frac{\omega}{2}{-}\epsilon\Big)\frac{\pi^2}{L^2}V(x) 
-2\hat{\alpha}(V(x))
+ \frac{1}{\omega} \|u\|^{2}_{L_{2}(0,L)}.
\label{eq:W_12q_Final_q_equals_2}
\end{align}
Here, the last inequality uses \eqref{eq:xihatal}. 
Recall that $\epsilon>0$. 
Pick $\omega\in(0,2\epsilon)$. 
Property \eqref{eq:Assumption_on_a} 
ensures that $V$ is an ISS Lyapunov function of \eqref{Nonlinear_Parabolic_W_12q} with respect to input space $U=L_{2q}(0,L)$. 
Application of Friedrichs' inequality to $\|u\|^{2}_{L_{2}(0,L)}$ 
proves that $V$ is an ISS Lyapunov function
with respect to $U=W_0^{1,2}(0,L)\cap W^{2,2}(0,L)=H_0^1(0,L)\cap H^2(0,L)$.

Next, assume that $q>1$. We apply Young's inequality \eqref{ineq:Young} to the last term in \eqref{eq:W_12q_Temporary_Estimate} with $\omega_2>0$ as follows
\begin{equation*}
\Big( \frac{\partial x}{\partial l} \Big)^{2q-2} u^2 
\leq
\frac{1}{q\omega_2} u^{2q} + \omega^{\frac{1}{q-1}}_2\frac{q-1}{q} 
\Big( \frac{\partial x}{\partial l} \Big)^{2q}
\end{equation*}
Putting this expression into \eqref{eq:W_12q_Temporary_Estimate} and using \eqref{eq:Assumption_xi} we obtain finally
\begin{align}
\tfrac{1}{2q (2q-1)} \dot{V} 
&\leq \Big( \big(\frac{\omega}{2}-c\big)\frac{\pi^2}{q^2L^2} +\omega^{\frac{1}{q-1}}_2 \frac{q-1}{2\omega q}\Big)V(x) 
\nonumber \\
& \ 
-\int_0^L \hspace{-1ex}\eta\left(\left(\tfrac{\partial x}{\partial l}\right)^{2q}\right) dl
+ \frac{1}{2\omega \omega_2 q} \|u\|^{2q}_{L_{2q}(0,L)} 
\nonumber \\
&\leq \Big( \big(\frac{\omega}{2}-\epsilon\big)\frac{\pi^2}{q^2L^2} +\omega^{\frac{1}{q-1}}_2 \frac{q-1}{2\omega q}\Big)V(x) 
\nonumber \\
& \quad - \hat{\alpha}(V(x)) + \frac{1}{2\omega \omega_2 q} \|u\|^{2q}_{L_{2q}(0,L)} .
\label{eq:W_12q_Final_Estimate}
\end{align}
It is easy to see that choosing $\omega>0$ and $\omega_2>0$ small enough, we can ensure  
$(\frac{\omega}{2}-\epsilon)\frac{\pi^2}{q^2 L^2}+\omega^{\frac{1}{q-1}}_2 \frac{q-1}{2\omega q}<0$. 
Hence, due to \eqref{eq:Assumption_on_a}, 
the function $V$ is an ISS Lyapunov function of \eqref{Nonlinear_Parabolic_W_12q} with respect to input space $U=L_{2q}(0,L)$. 
Application of Friedrichs' inequality to $\|u\|^{2q}_{L_{2q}(0,L)}$ 
proves that $V$ is an ISS Lyapunov function
with respect to $U=W_0^{1,2q}(0,L)\cap W^{2,2q}(0,L)$.


Finally, to prove Item \ref{item:ISSLFitem2} of the theorem, we 
are going to estimate the last term in \eqref{BasicEstimate_For_Other_U} with the help of 
\begin{align}
- \int_0^L\hspace{-1ex}\Big( \frac{\partial x}{\partial l} \Big)^{2q-2} 
\frac{\partial^2 x}{\partial l^2} u  dl
&=
- \frac{1}{2q\!-\!1}\hspace{-.4ex}\int_0^L\hspace{-1ex}\frac{\partial }{\partial l} \Big(\Big( \frac{\partial x}{\partial l} \Big)^{2q-1}\Big) u  dl \nonumber\\
&=
\frac{1}{2q\!-\!1} \int_0^L \Big( \frac{\partial x}{\partial l} \Big)^{2q-1} \frac{\partial u}{\partial l}  dl \nonumber\\
&\leq
\frac{1}{2q\!-\!1} \int_0^L \Big| \frac{\partial x}{\partial l} \Big|^{2q-1} \Big| \frac{\partial u}{\partial l}\Big|  dl . 
\label{eq:Temp_Estimate}
\end{align}
Here, $u(0)=u(L)=0$ is used in integration by parts. By virtue of \eqref{eq:Assumption_on_alpha} it holds that $sg(s) = \xi(s^{\frac{2q}{2q-1}})$ and applying 
inequality \eqref{ineq:Kinf} to the term $\Big| \frac{\partial x}{\partial l} \Big|^{2q-1} \Big| \frac{\partial u}{\partial l}\Big|$ with $g$ yields 
\begin{align}
\int_0^L \Big| \frac{\partial x}{\partial l} \Big|^{2q-1} \Big|
\frac{\partial u}{\partial l}\Big|  dl 
\leq 
\omega\int_0^L\xi\left(\Big| \frac{\partial x}{\partial l} \Big|^{2q}\right) dl
 + \int_0^L \Big| \frac{\partial u}{\partial l}\Big| g^{-1}\Big(\frac{1}{\omega}\Big| \frac{\partial u}{\partial l}\Big|\Big) dl
\label{eq:Temp_Estimate_2}
\end{align}
for $\omega>0$. 
From \eqref{BasicEstimate_For_Other_U}, \eqref{eq:LinTermEstimate} and
the definition of $\xi$ it follows that 
\begin{align}
\tfrac{1}{2q (2q-1)} \dot{V} 
\leq& - \frac{c\pi^2}{q^2L^2}V(x) 
-\int_0^L \eta\left(\left(\tfrac{\partial x}{\partial l}\right)^{2q}\right) dl
- \int_0^L \Big( \frac{\partial x}{\partial l} \Big)^{2q-2} \frac{\partial^2 x}{\partial l^2} u  dl 
\nonumber \\
=& - \frac{\epsilon\pi^2}{q^2L^2} V(x) 
-\int_0^L \xi\left(\left(\tfrac{\partial x}{\partial l}\right)^{2q}\right) dl
- \int_0^L \Big( \frac{\partial x}{\partial l} \Big)^{2q-2} \frac{\partial^2 x}{\partial l^2} u  dl .
\label{BasicEstimate_For_Other_Ub}
\end{align}
Substituting \eqref{eq:Temp_Estimate} and \eqref{eq:Temp_Estimate_2} into \eqref{BasicEstimate_For_Other_Ub} we obtain
\begin{align}
\tfrac{1}{2q (2q-1)} \dot{V}(x) 
\leq& -\frac{\epsilon\pi^2}{q^2L^2} V(x) 
-\left(\!1-\frac{\omega}{2q-1}\!\right)\int_0^L \xi\left(\left(\tfrac{\partial x}{\partial l}\right)^{2q}\right) dl
\nonumber\\
\quad& + \frac{1}{2q-1}\int_0^L \Big| \frac{\partial u}{\partial l}\Big| g^{-1}\Big(2\Big| \frac{\partial u}{\partial l}\Big|\Big) dl.
\nonumber 
\end{align}
Let $\omega < 2q-1$. In view of \eqref{eq:xihatal} this implies
\begin{align}
\tfrac{1}{2q (2q-1)} \dot{V}(x) 
&\leq -\frac{\epsilon\pi^2}{q^2L^2} V(x) 
-\left(\!1-\frac{\omega}{2q-1}\!\right)\hat{\alpha}(V(x)) 
\nonumber \\
& \quad+ \frac{1}{2q-1}\int_0^L \Big| \frac{\partial u}{\partial l}\Big| g^{-1}\Big(\frac{1}{\omega}\Big| \frac{\partial u}{\partial l}\Big|\Big) dl.
\label{Last_estim_altern_U}
\end{align}
Recall that $\epsilon\geq 0$. 
Property \eqref{eq:Assumption_on_alpha} satisfied with $g \in \Kinf$
implies $\hat{\alpha}\in\K_\infty$.  
Consequently \eqref{Last_estim_altern_U} means that 
$V$ is an ISS Lyapunov function of \eqref{Nonlinear_Parabolic_W_12q} with respect to input space $U=U_g$. 
\end{proof}

\begin{remark}
In the above proof, several times we have used integration by parts as well as partial derivatives of $x$ and $u$ and thus the derivations are justified if the functions are smooth enough. Notice that having established estimates of $\dot{V}$ for the spaces of smooth functions, which are dense subspaces of $X$ and $U$ respectively, the density argument as in \cite[Propositions 2,3, proof of Theorem 6]{MiI14b} ensures the result on the whole spaces $X$ and $U$.
\end{remark}

Items \ref{item:ISSLFitem1} and \ref{item:ISSLFitem2} 
of Theorem \ref{theorem:parabolic_ISS_W_12q_norm} 
demonstrate that different choices of input spaces result in 
different properties of a single system even if the state space 
is the same. 
Item \ref{item:ISSLFitem2} of Theorem 
\ref{theorem:parabolic_ISS_W_12q_norm} 
also illustrates that ISS does not necessarily imply an exponential decay rate. 
In contrast, for Item~\ref{item:ISSLFitem1} of 
Theorem \ref{theorem:parabolic_ISS_W_12q_norm}, 
the constructed function $V$ is guaranteed to 
exhibit an exponential or faster decay rate globally. 
As this is the case for finite dimensional systems, 
it is observed for infinite dimensional systems that 
according to \eqref{DissipationIneq} with $\alpha \in \PD$ which 
can be bounded, 
iISS allows the decay rate of $V$ to be much slower 
for large magnitude of state variables than ISS can allow. 
This indicates that significantly different constructions for iISS 
Lyapunov functions are needed. Next section is devoted to this question.


\section{{\rm i}ISS of a class of nonlinear parabolic systems: $L_p$ state space}
\label{sec:parablic}

Consider a system 
\begin{equation}
\label{GekoppelteNonLinSyst_ODE-PDE}
\frac{\partial x}{\partial t} = c \frac{\partial^2 x}{\partial l^2} + f(x(l,t),u(l,t)) ,\quad  \forall  t>0
\end{equation}
defined on the spatial domain $(0,L)$ with 
\begin{eqnarray}
x(0,t) \frac{\partial x}{\partial l}(0,t) = x(L,t) \frac{\partial x}{\partial l}(L,t)=0, 
\quad \forall t \geq 0 
\label{eq:BoundaryConditions_1eq}
\end{eqnarray}
which represents boundary conditions of Dirichlet, Neumann or mixed type. 
The state space for \eqref{GekoppelteNonLinSyst_ODE-PDE} we choose as $X=L_{2q}(0,L)$ for some $q \in \N$ and input space we take as $U=L_{\infty}(0,L)$ and $H^1_0(0,L)$.

Define the following ODE associated with 
\eqref{GekoppelteNonLinSyst_ODE-PDE} given by
\begin{eqnarray}
\dot{y}(t)=f(y(t),u(t)) , \quad y(t),u(t)\in\R.
\label{eq:ODEsys}
\end{eqnarray}
The next theorem provides a construction of an iISS Lyapunov function for a class of nonlinear systems of the form \eqref{GekoppelteNonLinSyst_ODE-PDE}.
\begin{Satz}
\label{theorem:parabplic_iISS}
Suppose that $W: y \mapsto y^{2q}$ satisfies 
\begin{eqnarray}
\label{eq:LF_for_ODEsysi}
\dot{W}(y): = 2q y^{2q-1} f(y,u) \leq - \alpha(W(y)) + W(y)\sigma(|u|)
\end{eqnarray}
for some $\alpha\in\Kinf\cup\{0\}$, $\sigma \in \K$.
Let any of the following 
conditions hold:
\begin{enumerate}
\item\label{item:iISS_boundary} 
$x(0,t)=0$ for all $t\geq 0$ or $x(L,t)=0$ for all $t\geq 0$.
\item\label{item:iISS_alpha} 
$\alpha$ is convex and $\Kinf$. 
\end{enumerate}
Then an iISS Lyapunov function of \eqref{GekoppelteNonLinSyst_ODE-PDE} with 
\eqref{eq:BoundaryConditions_1eq} with respect to the spaces of input values
$U=L_{\infty}(0,L)$ as well as $U=H^1_0(0,L)$ 
is given by
\begin{align}
V(x)=\ln(1+Z(x)) , 
\label{eq:LF_for_PDEi}
\end{align}
where $Z$ is defined as 
\begin{align}
Z(x)=\int_0^L W(x(l)) dl = \|x\|_{L_{2q}(0,L)}^{2q}.
\label{eq:LF_for_PDE}
\end{align}
Furthermore, if $\alpha$ is convex and satisfies 
\begin{align}
\label{eq:parabolic_alporder}
\liminf_{s\to\infty}\dfrac{\alpha(s)}{s}=\infty, 
\end{align}
then $V$ given above is 
an ISS Lyapunov system of \eqref{GekoppelteNonLinSyst_ODE-PDE} with 
\eqref{eq:BoundaryConditions_1eq} with respect to 
$U=L_{\infty}(0,L)$ as well as $U=H^1_0(0,L)$.
\end{Satz}
\begin{proof}
Consider $Z$ given by \eqref{eq:LF_for_PDE} and let $U=L_{\infty}(0,L)$.
Using \eqref{eq:LF_for_ODEsysi} 
we have
\begin{align}
\dot{Z}(x) =& 2q \int_0^L x^{2q-1}(l,t) 
\cdot \left(c \frac{\partial^2 x}{\partial l^2}(l,t) + f(x(l,t),u(l,t)) \right)dl 
\nonumber \\
\leq& -2q (2q-1)c \int_0^L x^{2q-2} \left(\frac{dx}{dl} \right)^2 dl 
\nonumber \\
& \quad\quad\quad\quad\quad+ \!\int_0^L \hspace{-1.2ex}\!\left\{-\alpha\big(W(x(l,t))\big) \!+\! W(x(l,t)) \sigma(|u(l,t)|)\!\right\} dl  
\nonumber \\
\leq& - \frac{2(2q-1)c}{q} \int_0^L \left(\frac{d}{dl}(x^q) \right)^2 dl 
-\int_0^L \hspace{-.2ex}\alpha\big(W(x(l,t))\big)dl 
 + Z(x) \sigma(\|u(\cdot,t)\|_{U}).
\label{eq:parabolic_iISS_V}
\end{align}
In the last estimate we have used boundary conditions \eqref{eq:BoundaryConditions_1eq}. 

First, suppose that Item \ref{item:iISS_boundary} holds. 
Since $x \in W^{2q,1}(0,L)$ then $x^q \in L_2(0,L)$ and $\frac{d}{dl}(x^q)=qx^{q-1}\frac{dx}{dl} \in L_2(0,L)$ due to H\"older's inequality (since $\frac{dx}{dl} \in L_{2q}(0,L)$). 
Overall, we have $x^q \in W^{2,1}(0,L)$. Applying Poincare's inequality to the first term in 
\eqref{eq:parabolic_iISS_V}, we obtain 
\begin{align*}
\dot{Z}(x) \leq - \frac{2(2q-1)c}{q} \frac{\pi^2}{4L^2}Z(x) + Z(x) \sigma(\|u(\cdot,t)\|_{U}) 
\end{align*}
with the help of $x(0,t)=0$ for all $t\geq 0$ or $x(L,t)=0$ for all $t\geq 0$. Defining $V$ as in \eqref{eq:LF_for_PDEi} results in 
\begin{align}
\dot{V}(x) \leq &- \frac{2(2q-1)c}{q} \frac{\pi^2}{4L^2} \frac{\|x\|_X^{2q}}{1+\|x\|_X^{2q}} + \sigma(\|u(\cdot,t)\|_{U}) . 
\label{Zderivative}
\end{align}
Recall that $X=L_{2q}(0,L)$.
Hence, Definition \ref{Def:iISSV} indicates that $V$ is 
an iISS Lyapunov function of \eqref{GekoppelteNonLinSyst_ODE-PDE} 
with boundary conditions \eqref{eq:BoundaryConditions_1eq} 
for the space $U=L_\infty(0,\infty)$. 

Next, assume that Item \ref{item:iISS_alpha} is satisfied. 
Due to the convexity of $\alpha$, Jensen's inequality in 
\eqref{eq:parabolic_iISS_V} allows us to obtain 
\begin{align*}
\dot{Z}(x) \leq 
-L\alpha\left(\frac{1}{L}Z(x(l,t))\right)
+   Z(x) \sigma(\|u(\cdot,t)\|_{U}) .
\end{align*}
For $V$ in \eqref{eq:LF_for_PDEi} we have 
\begin{align}
\dot{V}(x) \leq
-\frac{L\alpha\left(\frac{1}{L}\|x\|_X^{2q}\right)}{1+\|x\|_X^{2q}}
+ \sigma(\|u(\cdot,t)\|_{U}) .
\label{Zderivative2}
\end{align}
According to Definition \ref{Def:iISSV}, the function $V$ is 
an iISS Lyapunov function of \eqref{GekoppelteNonLinSyst_ODE-PDE} 
with boundary conditions \eqref{eq:BoundaryConditions_1eq} 
for the space of input values $U=L_\infty(0,\infty)$. 
Since \eqref{eq:parabolic_alporder} implies 
$\liminf_{s\to\infty}{\alpha(s)}/({1+Ls})=\infty$, 
the above inequality guarantees that $V$ is 
an ISS Lyapunov function in the case of \eqref{eq:parabolic_alporder}. 

Finally, 
to deal with the space $H^1_{0}(0,\pi)$ for the input values, 
we recall Agmon's inequality \eqref{ineq:Agmon}, which implies for $u \in H^1_0(0,\pi)$
\begin{align*}
\|u\|^2_{L_{\infty}(0,L)} \leq 
\|u\|^2_{L_{2}(0,L)} + 
\Big\|\frac{\partial u}{\partial l}\Big\|^2_{L_{2}(0,L)}.
\end{align*}
This inequality yields 
\begin{align}
\|u\|^2_{L_{\infty}(0,L)} \leq 
\left(\frac{L^2}{\pi^2}+1\right) \|u\|^2_{H^1_{0}(0,L)}.
\label{eq:agmonfried0}
\end{align}
with the help of Friedrichs' inequality \eqref{ineq:Friedrichs}. 
Substitution of \eqref{eq:agmonfried0} into 
\eqref{Zderivative} and \eqref{Zderivative2}, 
proves that $V$ is 
an iISS Lyapunov function of 
\eqref{GekoppelteNonLinSyst_ODE-PDE}-\eqref{eq:BoundaryConditions_1eq} 
with respect to $U=H^1_{0}(0,\pi)$ under either Item \ref{item:iISS_boundary} or Item \ref{item:iISS_alpha}.
\end{proof}

\begin{remark}
We want to stress a reader's attention on the choice of an input space. First we have proved iISS of the system \eqref{eq:bilinear1} for the input space $L_{\infty}(0,L)$. For many applications this choice of input space is reasonable and sufficient. However, when considering interconnections of control systems, the input to one system is a state of another system. Thus, having $L_{\infty}(0,L)$ as an input space of the first subsystem automatically means that it is a state space of another subsystem, which complicates the proof its ISS, since the constructions of Lyapunov functions for this choice of state space are hard to find (e.g. how to differentiate such Lyapunov functions?), if possible. As we have seen in Section~\ref{sec:ISS_Parabolic}, this is not the case if we choose $H^1_0(0,L)$ as a state space. This underlines the role of the Agmon's inequality in our constructions, which made possible the transition from the space $L_{\infty}(0,L)$ to $H^1_0(0,L)$ in the previous theorem. 
\end{remark}

Note that the term $W(y)\sigma(|u|)$ in \eqref{eq:LF_for_ODEsysi} allows to analyze PDEs \eqref{GekoppelteNonLinSyst_ODE-PDE} with bilinear or generalized bilinear terms which do not possess ISS property. 


\section{{\rm i}ISS of a class of nonlinear parabolic systems: Sobolev state space}
\label{sec:parablicsobolev}

Instead of the $L_2$ state space we used for characterizing iISS
in Section \ref{sec:parablic}, for a class of parabolic systems, 
this section demonstrates that iISS can be established 
with Sobolev state space. 
We consider 
\begin{align}
\frac{\partial x}{\partial t} = c \frac{\partial^2 x}{\partial l^2} + f\big(x(l,t),\tfrac{\partial x}{\partial l}(l,t)\big)
+ \dfrac{\partial x}{\partial l}(l,t)u(l,t)
\label{eq:bilinear1}
\end{align}
defined for $(l,t) \in (0,L) \times (0,\infty)$ with the Dirichlet boundary conditions
\begin{eqnarray}
x(0,t) = x(L,t) =0, \quad \forall t \geq 0 .
\label{eq:bilinear1boundary}
\end{eqnarray}
We take $X=W^{1,2q}_0(0,L)$, $q \in \N$. 
Modifying Item \ref{item:ISSLFitem1} of 
Theorem \ref{theorem:parabolic_ISS_W_12q_norm}, 
we can verify the following. 

\begin{Satz}\label{theorem:bilinear1}
Suppose that 
\eqref{eq:Assumption_on_f} holds for all $x \in X$ with 
some convex continuous function $\eta: \R_+\to \R$ 
and some $\epsilon\in\R_+$ such that 
\eqref{eq:Assumption_on_a} holds. 
If $\epsilon>0$, then 
the function $V$ given by 
\begin{align}
&
V(x)=\ln(1+Z(x)) , 
\label{eq:bilinear1V}
\\
&
Z(x)=\int_0^L \Big( \frac{\partial x}{\partial l} \Big)^{2q} dl = \|x\|_{W^{1,2q}_0(0,L)}^{2q} 
\label{eq:bilinear1Z}
\end{align}
is an iISS Lyapunov function of 
\eqref{eq:bilinear1}-\eqref{eq:bilinear1boundary} with respect to the space $U=L_\infty(0,L)$ of input values and $U=H^1_0(0,L)$ as well. 
\end{Satz}
\begin{proof}
As in the proof of Theorem \ref{theorem:parabolic_ISS_W_12q_norm}, 
in the case of $q=1$, along solutions of 
\eqref{eq:bilinear1}-\eqref{eq:bilinear1boundary},  
the function $Z$ in \eqref{eq:LF_W_12q_norm} satisfies 
\begin{align}
\dot{Z} 
\leq & 
2\Big(\frac{\omega}{2}-\epsilon\Big)\frac{\pi^2}{L^2}Z(x) 
-2\hat{\alpha}(Z(x))
+ \frac{1}{\omega}\int_0^L \Big( \frac{\partial x}{\partial l} \Big)^{\!2} u^2 dl.
\nonumber \\
\leq & 
2\Big(\frac{\omega}{2}-\epsilon\Big)\frac{\pi^2}{L^2}Z(x) 
-2\hat{\alpha}(Z(x))
+ \frac{Z(x)}{\omega} \|u\|^{2}_{L_\infty(0,L)}
\label{eq:bilinear1vdot1}
\end{align}
for any $\omega>0$. 
Due to \eqref{eq:bilinear1V} we have 
\begin{align}
\dot{V} 
\leq & 
2\Big(\frac{\omega}{2}-\epsilon\Big)\frac{\pi^2}{L^2}
\frac{\|x\|_X^{2}}{1+\|x\|_X^{2}}
-\frac{2\hat{\alpha}(\|x\|_X^2)}{1+\|x\|_X^{2}}
+ \frac{1}{\omega} \|u\|^{2}_{L_\infty(0,L)}.
\label{eq:bilinear1vdot2}
\end{align}
Pick $\omega\in(0,2\epsilon)$. Then $\epsilon >0$ and property 
\eqref{eq:bilinear1vdot2} imply that 
$V$ is an iISS Lyapunov function with respect to the space 
$U=L_\infty(0,L)$ of input values.

Next consider $q>1$. Again, following the argument used in 
the proof of Theorem \ref{theorem:parabolic_ISS_W_12q_norm}, 
we obtain
\begin{align}
\tfrac{1}{2q (2q-1)} \dot{Z} 
&\leq 
\Big( \big(\frac{\omega}{2}-\epsilon\big)\frac{\pi^2}{q^2L^2} +\omega^{\frac{1}{q-1}}_2 \frac{q-1}{2\omega q}\Big)Z(x) 
\nonumber \\
& \quad - \hat{\alpha}(Z(x)) + \frac{Z(x)}{2\omega \omega_2 q} 
\|u\|^{2q}_{L_\infty(0,L)} .
\label{eq:bilinear1vdot3}
\end{align}
for any $\omega, \omega_2>0$. 
From  \eqref{eq:bilinear1V} it follows that 
\begin{align}
\tfrac{1}{2q (2q-1)} \dot{V} 
&\leq 
\Big( \big(\frac{\omega}{2}-\epsilon\big)\frac{\pi^2}{q^2L^2} +\omega^{\frac{1}{q-1}}_2 \frac{q-1}{2\omega q}\Big)
\frac{\|x\|_X^{2q}}{1+\|x\|_X^{2q}}
\nonumber \\
& \quad - 
\frac{\hat{\alpha}(\|x\|_X^{2q})}{1+\|x\|_X^{2q}}
+ \frac{1}{2\omega \omega_2 q} 
\|u\|^{2q}_{L_\infty(0,L)} .
\label{eq:bilinear1vdot4}
\end{align}
This inequality with sufficiently small $\omega, \omega_2>0$ 
implies that 
$V$ is an iISS Lyapunov function with respect to the space 
$U=L_\infty(0,L)$ of input values. 

To deal with the space $H^1_{0}(0,\pi)$ for the input values, 
substitute \eqref{eq:agmonfried0} into 
\eqref{eq:bilinear1vdot2} and \eqref{eq:bilinear1vdot4}. 
\end{proof}


\section{Interconnections of {\rm i}ISS systems}
\label{GekoppelteISS_Systeme}


Consider the following interconnected system: 
\begin{align}
\label{eq:interconnection}
\begin{array}{l}
\dot{x}_i(t)=A_ix_i(t)+f_i(x_1,x_2,u) , \quad i=1,2 \\
x_i(t)\in X_i , \quad u\in U_c , 
\end{array}
\end{align}
where $X_i$ is a state space of the $i$-th subsystem, $A_i:D(A_i) \to X_i$ is a generator of a strongly continuous semigroup over $X_i$. 
Let $X=X_1\times X_2$ which is the space of $x=(x_1, x_2)$, 
and the norm in $X$ is defined as 
$\|\cdot\|_X=\|\cdot\|_{X_1}+\|\cdot\|_{X_2}$. 
In this section, we assume that 
there exist continuous functions $V_i:X_i \to \R_+$, 
$\psi_{i1},\psi_{i2} \in \Kinf$, $\alpha_i \in \PD$, 
$\sigma_i \in \K$  and $\kappa_i \in \K\cup\{0\}$ 
for $i=1,2$ such that 
\begin{equation}
\label{LyapFunk_1Eigi}
\psi_{i1}(\|x_i\|_{X_i}) \leq V_i(x_i) \leq \psi_{i2}(\|x_i\|_{X_i}), 
\quad \forall x_i \in X_i
\end{equation}
and system \eqref{eq:interconnection} satisfies 
\begin{equation}
\label{GainImplikationi}
\dot{V}_i(x_i) \leq -\alpha_i(\|x_i\|_{X_i})
 + \sigma_i(\|x_{3-i}\|_{X_{3-i}}) + \kappa_i(\|u(0)\|_U)
\end{equation}
for all $x_i\in X_i$, $x_{3-i}\in X_{3-i}$ and $u\in U_c$, 
where the Lie derivative of $V_i$ corresponding to the inputs 
$u\in U_c$ and $v\in PC(\R_+,X_{3-i})$ with $v(0)=x_{3-i}$ 
is defined by
\begin{equation}
\label{LyapAbleitungi}
\dot{V}_i(x_i)=\mathop{\overline{\lim}} \limits_{t \rightarrow +0} {\frac{1}{t}(V_i(\phi_i(t,x_i,v,u))-V_i(x_i)) }.
\end{equation}

To present a small-gain criterion for the interconnected 
system \eqref{eq:interconnection} whose components are not 
necessarily ISS, we make use of 
an generalized expression of inverse mappings on the set of 
extended non-negative numbers $\overline{\R}_+=[0,\infty]$. 
For $\omega\in\K$, define the function $\omega^\ominus$: 
$\overline{\R}_+\to\overline{\R}_+$ as 
$\omega^\ominus(s)=\sup \{ v \in \R_+ : s \geq \omega(v) \}$. 
Notice that $\omega^\ominus(s)=\infty$ holds for 
$s\geq \lim_{\tau\to\infty}\omega(\tau)$, and 
$\omega^\ominus(s)=\omega^{-1}(s)$ holds elsewhere. 
A function $\omega\in\K$ is extended to 
$\omega$: $\overline{\R}_+\to\overline{\R}_+$ as 
$\omega(s) := \sup_{v\in\{y\in\R_+ \, : \, y \leq s\}} \omega(v)$. 
These notations are useful for presenting the following result 
succinctly. 

\begin{Satz}\label{theorem:intercon}
Suppose that 
\begin{align}
\label{eq:alpsig}
\displaystyle\lim_{s\rightarrow\infty}\!\alpha_i(s)=\infty
\ \mbox{or} \ 
\lim_{s\rightarrow\infty}\!\sigma_{3-i}(s)\kappa_i(1)<\infty
\end{align}
is satisfied for $i=1,2$. 
If there exists $c>1$ such that 
\begin{align}
\label{eq:sg}
\psi_{11}^{-1}\circ\psi_{12}\circ
\alpha_1^\ominus\circ c\sigma_1\circ
\psi_{21}^{-1}\circ\psi_{22}\circ
\alpha_2^\ominus\circ c\sigma_2(s)
\leq s
\end{align}
holds for all $s\in\R_+$, then system \eqref{eq:interconnection} is iISS. 
Moreover, if additionally $\alpha_i\in\Kinf$ for $i=1,2$, then system \eqref{eq:interconnection} is ISS. 
Furthermore, 
\begin{align}
\label{eq:Vsum}
V(x)=\int_0^{V_1(x_1)}\lambda_1(s)ds
+\int_0^{V_2(x_2)}\lambda_2(s)ds
\end{align}
is an iISS (ISS) Lyapunov function for \eqref{eq:interconnection}, where $\lambda_i\in\K$ 
is given for $i=1,2$ by 
\begin{align}
\label{eq:lambda}
\lambda_i(s)=[\alpha_i(\psi_{i2}^{-1}(s))]^\psi
[\sigma_{3-i}(\psi_{3-i\,1}^{-1}(s))]^{\psi+1} , \forall s\in \R_+
\end{align}
with an arbitrary $\psi\geq 0$ satisfying 
\begin{align}
\label{eq:tauphi}
\begin{array}{ll}
\psi=0 &, \mbox{if } c> 2 
\\
\displaystyle\psi^{-\frac{\psi}{\psi+1}}< \frac{c}{\psi+1}\leq 1
&, \mbox{otherwise. }
\end{array}
\end{align}
\end{Satz}
\begin{proof}
For the continuous function $V:X \to \R_+$ given by \eqref{eq:Vsum}, 
we have 
\begin{align*}
\dot{V} \leq & \sum_{i=1}^2\lambda_i(V_i)
\bigl\{-\alpha_i(\psi_{i2}^{-1}(V_i(x_i)))
\\
& \hspace{11ex}
+ \sigma_i(\psi_{3-i\,1}^{-1}(V_{3-i}(x_{3-i}))) + \kappa_i(\|u\|_U)\bigr\}
\end{align*}
along the solution of \eqref{eq:interconnection}, 
due to \eqref{LyapFunk_1Eigi} and \eqref{GainImplikationi}. 
Following the arguments used in \cite{ItJ09,IJD13}, we can verify that 
with \eqref{eq:lambda} and \eqref{eq:tauphi}, the property 
\begin{align*}
\dot{V} \leq & \sum_{i=1}^2
\bigl\{-\delta_i\lambda_i(\psi_{i1}(V_i(x_i)))\alpha_i(\psi_{i2}^{-1}(\psi_{i1}(V_i(x_i)))
\nonumber \\
& \hspace{11ex}+\hat{\kappa}_i(\|u\|_U)\bigr\}
\end{align*}
holds for some $\hat{\kappa_1}, \hat{\kappa}_2\in\K\cup\{0\}$ and constants $\delta_2,\delta_1>0$ 
if \eqref{eq:sg} and \eqref{eq:alpsig} are satisfied. 
In addition, we have $\hat{\kappa}_i=0$ if $\kappa_i=0$. 
Hence, Proposition \ref{PropSufiISS} completes the proof 
with the help of the definition of $\|\cdot\|_X$ and $\lambda_i\in\K$. 
\end{proof}

It is straightforward to see that there always exists 
$\psi\geq 0$ satisfying \eqref{eq:tauphi}. 
It is also worth mentioning that the Lyapunov function \eqref{eq:Vsum} 
is not in the maximization form, employed in \cite{DaM13} 
for establishing ISS. 
The use of the summation form \eqref{eq:Vsum} for systems which 
are not necessarily ISS is motivated by the limitation of 
the maximization form and clarified in \cite{IDW12} 
for finite-dimensional systems. 

\section{Examples}\label{sec:Example}

This section puts the results presented in the preceding sections 
together to analyze two reaction-diffusion systems.

\subsection{Example 1}\label{ssec:Example1}

Consider 
\begin{equation}
\label{GekoppelteNonLinSyst}
\left
\{
\begin{array}{l} 
{\dfrac{\partial x_1}{\partial t}(l,t) = \dfrac{\partial^2 x_1}{\partial l^2}(l,t) + x_1(l,t)x^4_2(l,t),} \\[1.5ex]
{x_1(0,t) = x_1(\pi,t)=0;} \\[1ex]
{\dfrac{\partial x_2}{\partial t} =  \dfrac{\partial^2 x_2}{\partial l^2} + ax_2 - bx_2 \Big(\dfrac{\partial x_2}{\partial l}\Big)^{\!\!2} + \Big( \frac{x^2_1}{1+x^2_1} \Big)^{\!\frac{1}{2}}}\hspace{-1ex}, \\
{x_2(0,t) = x_2(\pi,t)=0.} 
\end{array}
\right.	
\end{equation}
defined on the region $(l,t) \in (0,\pi) \times (0,\infty)$. 
To fully define the system we should choose the state spaces of subsystems. We take $X_1:=L_2(0,\pi)$ for $x_1(\cdot,t)$ and $X_2:=H^1_0(0,\pi)$ for $x_2(\cdot,t)$ 
as in Table \ref{table:iISS} (a).  
We divide the analysis into three parts. First we prove that 
$x_1$-subsystem is iISS based on the development in 
Sections \ref{sec:parablic}. 
Next we prove that the $x_2$-subsystem is ISS using the result 
in Section \ref{sec:ISS_Parabolic}. 
In the last part we exploit the small-gain theorem presented 
in Section \ref{GekoppelteISS_Systeme} 
to prove UGASs of $x=0$ of the overall system \eqref{GekoppelteNonLinSyst}.

\subsubsection{The first subsystem is iISS}\label{sssec:exsys1}

First we invoke Item \ref{item:iISS_boundary} of 
Theorem~\ref{theorem:parabplic_iISS} 
with $q=1$ for $X_1=L_2(0,\pi)$. 
Then $W(y)=y^2$ and due to $\dot{W}(y) \leq 2W(y) |x_2|^4$ 
and $x_1(0,t) = x_1(\pi,t)=0$ for all $t\in\R_+$, 
one can choose
\begin{eqnarray}
V_1(x_1):=\ln\Big(1+\|x_1\|^2_{L_{2}(0,\pi)} \Big)
\label{eq:V1}
\end{eqnarray}
as an iISS Lyapunov function for $x_1$-subsystem. Its Lie derivative according to \eqref{Zderivative} satisfies 
\begin{eqnarray}
\dot{V}_1 \leq  
- \frac{2\|x_1\|_{L_2(0,\pi)}^2}{1+\|x_1\|_{L_2(0,\pi)}^2} 
+ 2\|x_2\|_{L_\infty(0,\pi)}^4 . 
\label{eq:Z1dot}
\end{eqnarray}
Note, that we have put $1=\frac{\pi^2}{\pi^2}$ instead of $\frac{\pi^2}{4\pi^2}$ in formula \eqref{Zderivative} because $x_1=0$ holds at both ends of the interval $[0,\pi]$, and thus less conservative  Friedrichs' inequality instead of Poincare's inequality can be used in getting \eqref{Zderivative}. 
To replace $L_\infty(0,\pi)$ with $X_2=H^1_0(0,\pi)$ 
for the input space used in \eqref{eq:Z1dot}, 
we recall \eqref{eq:agmonfried0}, which results in 
\begin{eqnarray}
\dot{V}_1(x_1) \leq 
- \frac{2\|x_1\|_{L_2(0,\pi)}^2}{1+\|x_1\|_{L_2(0,\pi)}^2} 
+ 8  \|x_2\|^4_{H^1_{0}(0,\pi)}.
\label{eq:Z1dot_Final}
\end{eqnarray}
Thus, we arrive at \eqref{GainImplikationi} for $i=1$, and 
$x_1$-subsystem is iISS with respect to the state space 
$X_1=L_2(0,\pi)$ and the input space $X_2=H^1_0(0,\pi)$.

\subsubsection{The second subsystem is ISS}\label{sssec:exsys2}

We invoke Item \ref{item:ISSLFitem1} of 
Theorem~\ref{theorem:parabolic_ISS_W_12q_norm} with $q=1$. 
To simplify notation we denote 
$u_2:=( {x^2_1}/(1+x^2_1) )^{1/2}$. 
As in \eqref{eq:LF_W_12q_norm}, we take 
\begin{eqnarray}
V_2(x_2)=  \int_0^{\pi}{\Big(\frac{\partial x_2}{\partial l}\Big)^2  dl} = 
 \|x_2\|^2_{H^1_0(0,\pi)} .
\label{eq:LF_H10_Example}
\end{eqnarray}
Notice that $x_2$-subsystem is of the form \eqref{Nonlinear_Parabolic_W_12q} 
with $c=1$, $f(x_2,\tfrac{\partial x_2}{\partial l})=ax_2 - bx_2 (\tfrac{\partial x_2}{\partial l})^2$.
To arrive at \eqref{eq:Assumption_on_f}, we obtain 
\begin{align}
\int_0^L \frac{\partial^2 x_2}{\partial l^2}\Big( ax_2 - bx_2 \Big(\frac{\partial x_2}{\partial l}\Big)^2 \Big) dl
=  - a V_2(x_2) - b\hspace{-.5ex}\int_0^L\hspace{-.5ex}\frac{\partial^2 x_2}{\partial l^2}x_2 \Big(\frac{\partial x_2}{\partial l}\Big)^2 dl 
\end{align}
by integration by parts with 
$x_2(0,t) = x_2(\pi,t)=0$ for all $t\in\R_+$. 
Due to $x_2(0,t) = x_2(\pi,t)=0$ for all $t\in\R_+$, we have
\begin{align*}
\int_0^\pi{ \frac{\partial^2 x_2}{\partial l^2}  x_2 \Big(\frac{\partial x_2}{\partial l}\Big)^2 dl}
=&
- \int_0^\pi{ \frac{\partial x_2}{\partial l} \Big( 2  x_2 \frac{\partial x_2}{\partial l}\frac{\partial^2 x_2}{\partial l^2} +  \Big(\frac{\partial x_2}{\partial l}\Big)^3  \Big) dl} 
\nonumber \\
=&
- \int_0^\pi{  2  x_2 \Big(\frac{\partial x_2}{\partial l}\Big)^2\frac{\partial^2 x_2}{\partial l^2} dl}
- \int_0^\pi{ \Big(\frac{\partial x_2}{\partial l}\Big)^4 dl}, 
\end{align*}
which implies that
\begin{eqnarray}
\int_0^\pi{ \frac{\partial^2 x_2}{\partial l^2}  x_2 \Big(\frac{\partial x_2}{\partial l}\Big)^2 dl}
=
- \frac{1}{3} \int_0^\pi{ \Big(\frac{\partial x_2}{\partial l}\Big)^4 dl}.
\label{Useful_Expression}
\end{eqnarray}
Thus, we arrive at 
\begin{align}
\eta(s)=-as+\frac{b}{3}s^2 , \quad \forall s\in\R_+  
\end{align}
which is convex if $b\geq 0$. We also obtain 
\begin{align}
\hat{\alpha}(s)=(1-a-\epsilon)s+\frac{b}{3\pi}s^2
\label{eq:exhatalpha}
\end{align}
for \eqref{eq:Assumption_on_a}. The inequality in 
\eqref{eq:Assumption_on_a} is achieved for $\epsilon=1-a>0$ if 
$a<1$. Hence, if $a<1$ and $b\geq 0$ hold, 
Item \ref{item:ISSLFitem1} of 
Theorem~\ref{theorem:parabolic_ISS_W_12q_norm} with $q=1$ 
proves that for $\omega\in(0,2(1-a)]$, the function 
$V_2$ satisfies (with $\epsilon = 1-a$)
\begin{align}
\dot{V}_2 \leq -2(1-a-\frac{\omega}{2})V_2(x_2)
-\frac{2b}{3\pi}V_2(x_2)^2
+ \frac{1}{\omega} \|u\|^{2}_{L_{2}(0,\pi)}, 
\label{eq:x2issumod}
\end{align}
as in \eqref{eq:W_12q_Final_q_equals_2}, 
and $V_2$ is an ISS Lyapunov function of $x_2$-subsystem 
with respect to the state space 
$X_2=H^1_0(0,\pi)$ for $x_2(\cdot,t)$ and the input space 
$U_1=L_{2}(0,\pi)$ for $u_2(\cdot,t)$. 
Since $s \mapsto {s}/({1+s})$ is a concave function of 
$s\in\R_+$, Jensen's inequality yields 
\begin{eqnarray*}
\int_0^\pi\hspace{-.7ex}{ \frac{x^2_1}{1+x^2_1} dl} 
\leq \pi \frac{ (1/\pi)\|x_1\|^2_{L_{2}(0,\pi)}}{1+(1/\pi)\|x_1\|^2_{L_{2}(0,\pi)}} 
\leq
\frac{ \pi\|x_1\|^2_{L_{2}(0,\pi)}}{1+\|x_1\|^2_{L_{2}(0,\pi)}} . 
\end{eqnarray*}
Using this property in \eqref{eq:x2issumod} we have 
\begin{align}
\dot{V}_2 \leq 
-2(1-a-\frac{\omega}{2})\|x_2\|^2_{H^1_0(0,\pi)}
-\frac{2b}{3\pi}\|x_2\|^4_{H^1_0(0,\pi)}
+ \frac{\pi}{\omega}\left( 
\frac{\|x_1\|^2_{L_{2}(0,\pi)}}{1+\|x_1\|^2_{L_{2}(0,\pi)}}\right) . 
\label{eq:x2iss}
\end{align}
Therefore, $V_2$ is an ISS Lyapunov function of $x_2$-subsystem 
with respect to the state space 
$X_2=H^1_0(0,\pi)$ for $x_2(\cdot,t)$ and the input space 
$X_1=L_2(0,\pi)$ for $x_1(\cdot,t)$.

Although property \eqref{eq:x2iss} is satisfactory for 
establishing UGASs of the overall 
system \eqref{GekoppelteNonLinSyst}, it may be 
worth mentioning that 
we can obtain a different estimate for 
$\dot{V}_2$ using Item \ref{item:ISSLFitem2} of 
Theorem~\ref{theorem:parabolic_ISS_W_12q_norm}. 
Pick $\epsilon=1-a$. We obtain $\hat{\alpha}(s)=({b}/{3\pi})s^2$ 
from \eqref{eq:exhatalpha}. Then inequality 
\eqref{eq:Assumption_on_a} and  $\epsilon\geq 0$ hold 
if $a\leq 1$ and $b\geq 0$. 
Furthermore, if $b>0$, 
$g:s \mapsto  ({b}/{3}) s^3 $ satisfies
\eqref{eq:Assumption_on_alpha} with $q=1$ 
and it is of class $\Kinf$. Hence 
using $\omega=1/2$ and \eqref{Last_estim_altern_U}, 
we arrive at 
\begin{align}
\dot{V}_2\leq -2(1-a)\|x_2\|^2_{H^1_0(0,\pi)} 
-\frac{b}{3\pi}\|x_2\|^4_{H^1_0(0,\pi)}
+ 
2\left(\frac{6}{b}\right)^{\frac{1}{3}}
\int_0^\pi \Big| \frac{\partial u}{\partial l}\Big|^{4/3} dl , 
\end{align}
which implies that 
$V_2$ is an ISS Lyapunov function 
of $x_2$-subsystem with respect to the input space 
$U_2=W^{1,\frac{4}{3}}_0(0,\pi)$ 
if $a\leq 1$ and $b>0$.

Finally, it is worth noting that if $a >1$,  then 
$x=0$ of the linearization 
of $x_2$-subsystem for $x_1\equiv 0$ is not UGASs (see \cite[Theorem 5.1.3]{Hen81}).  Thus 
$x=0$ of the 
nonlinear $x_2$-subsystem for $x_1\equiv 0$ also cannot be 
UGASs if $a >1$.

\subsubsection{Interconnection is UGASs}

Now we collect the findings of two previous subsections. 
Assume that $a<1$ and $b\geq 0$. 
For the space $X=L_2(0,\pi)\times H^1_0(0,\pi)$, 
the Lyapunov functions defined as \eqref{eq:V1} and 
\eqref{eq:LF_H10_Example} for the two subsystems satisfy 
\eqref{LyapFunk_1Eigi} 
with the class $\K_\infty$ functions 
$\psi_{11}=\psi_{12}: s \mapsto \ln(1+s^2)$ and $\psi_{21}=\psi_{22}: s \mapsto s^2$. 
Due to \eqref{eq:Z1dot_Final} and \eqref{eq:x2iss}, 
we have \eqref{GainImplikationi} for 
\begin{align}
&
\alpha_1(s)=\frac{2s^2}{1+s^2} 
, \quad  
\sigma_1(s)=8s^4
, \quad  
\kappa_1(s)=0
\\
&
\alpha_2(s)=2\left(1\!-\!a\!-\!\frac{\omega}{2}\right)\!s^2 + \frac{2b}{3\pi}\!s^4
, \hspace{1ex} 
\sigma_2(s)=\frac{\pi}{\omega}\left(\!\frac{s^2}{1+s^2}\!\right), \;
\kappa_2(s)=0 
\end{align}
defined with $\omega\in(0,2(1-a)]$. 
For these functions, condition \eqref{eq:sg} holds 
for all $s\in\R_+$ if and only if 
\begin{align}
\frac{12c^2\pi^2}{b\omega}\left(\!\frac{s^2}{1+s^2}\!\right) 
\leq 
\frac{2s^2}{1+s^2} 
, \quad \forall s\in\R_+
\label{eq:exsg}
\end{align}
is satisfied. Thus, there exists $c>1$ such that \eqref{eq:sg} holds 
if and only if ${6\pi^2}/{b}<\omega$ holds. 
Combining this with $\omega\in(0,2(1-a)]$, $a<1$ and $b\geq 0$, 
Theorem \ref{theorem:intercon} establishes UGASs
of $x=0$ for the whole system \eqref{GekoppelteNonLinSyst} when 
\begin{align}
a+\frac{3\pi^2}{b}<1 , \quad  b\geq 0. 
\label{eq:expibassum}
\end{align}
Note that \eqref{eq:alpsig} is satisfied. 
Due to the boundary conditions of $x_2$, Friedrichs' inequality ensures 
$\|x_2(\cdot,t)\|_{L_2(0,\pi)}\leq \|x_2(\cdot,t)\|_{H^1_0(0,\pi)}$. 
Thus, the UGASs guarantees the existence of $\beta\in\KL$ 
such that 
\begin{align}
\left\| \phi(t,\phi_0,0) \right\|_{L_2(0,\pi)\times L_2(0,\pi)}
\leq \left\| \phi(t,\phi_0,0) \right\|_X 
\leq  \beta(\left\| \phi_0 \right\|_{X},t) 
\label{eq:exxestimate}
\end{align}
holds for all $\phi_0 \in X$ and all $t\in\R_+$, 
where $X=L_2(0,\pi)\times H^1_0(0,\pi)$. 


\subsection{Example 2}\label{ssec:Example2}

Consider 
\begin{equation}
\label{GekoppelteNonLinSyst2}
\left
\{
\begin{array}{l} 
{\dfrac{\partial x_1}{\partial t}(l,t) = \dfrac{\partial^2 x_1}{\partial l^2}(l,t) + \dfrac{\partial x_1}{\partial t}(l,t)x^4_2(l,t),} \\[1.5ex]
{x_1(0,t) = x_1(\pi,t)=0;} \\[1ex]
{\dfrac{\partial x_2}{\partial t} =  \dfrac{\partial^2 x_2}{\partial l^2} + ax_2 - bx_2 \Big(\dfrac{\partial x_2}{\partial l}\Big)^{\!\!2} + \Big( \frac{x^2_1}{1+x^2_1} \Big)^{\!\frac{1}{2}}}\hspace{-1ex}, \\
{x_2(0,t) = x_2(\pi,t)=0.} 
\end{array}
\right.	
\end{equation}
defined on the region $(l,t) \in (0,\pi) \times (0,\infty)$. 
For \eqref{GekoppelteNonLinSyst2}, 
we take $X_1:=H^1_0(0,\pi)$ and $X_2:=H^1_0(0,\pi)$ 
as in Table \ref{table:iISS} (b).

\subsubsection{The first subsystem is iISS}\label{sssec:ex2sys1}

We apply Theorem \ref{theorem:bilinear1} to $x_1$-subsystem on 
$X_1=H^1_0(0,\pi)$ by taking $q=1$.  
Let $V_1(x_1)$ be 
\begin{align}
V_1(x_1):=\ln\Big(1+\|x_1\|^2_{H^1_0(0,\pi)} \Big)
\label{eq:V1_2}
\end{align}
We can use $\eta=0$ for \eqref{eq:Assumption_on_f}, which 
is convex on $\R_+$. Let $\epsilon=c=1$. Then 
Property \eqref{eq:Assumption_on_a}
holds with $\hat{\alpha}=0$. 
From \eqref{eq:bilinear1vdot2} with $\omega=1$ 
and \eqref{eq:agmonfried0} it follows that 
\begin{align}
\dot{V}_1 \leq 
-\frac{\|x_1\|_{H^1_0(0,\pi)}^{2}}{1+\|x_1\|_{H^1_0(0,\pi)}^{2}}
+ 4\|x_2\|_{H^1_0(0,\pi)}^4.
\label{eq:exampleV1dot_2}
\end{align}
Hence, inequality \eqref{GainImplikationi} is obtained for $i=1$, 
and the $V_1$ in \eqref{eq:V1_2} is an iISS Lyapunov function 
of $x_1$-subsystem with respect to the state space 
$X_1=H^1_0(0,\pi)$ and the input space $X_2=H^1_0(0,\pi)$. 

\subsubsection{The second subsystem is ISS}\label{sssec:ex2sys2}

Since $x_2$-subsystem of \eqref{GekoppelteNonLinSyst2}
is identical with that of \eqref{GekoppelteNonLinSyst}, 
If $a<1$ and $b\geq 0$ holds, 
the function $V_1$ in \eqref{eq:LF_H10_Example} is 
an ISS Lyapunov function and satisfies \eqref{eq:x2iss} 
with respect to the state space 
$X_2=H^1_0(0,\pi)$ and the input space $X_1=H^1_0(0,\pi)$. 

\subsubsection{Interconnection is UGASs}

The above analysis for system \eqref{GekoppelteNonLinSyst2} 
yields \eqref{LyapFunk_1Eigi} and \eqref{GainImplikationi} for 
$i=1,2$, with functions which are the same as 
those for \eqref{GekoppelteNonLinSyst} except the following change:
\begin{align}
\alpha_1(s)=\frac{s^2}{1+s^2} 
, \quad  
\sigma_1(s)=4s^4. 
\end{align}
Again, condition \eqref{eq:sg} holds 
for all $s\in\R_+$ if and only if \eqref{eq:exsg} is satisfied. 
Hence, Theorem \ref{theorem:intercon} establishes UGASs
of $x=0$ for the whole system \eqref{GekoppelteNonLinSyst2} 
if \eqref{eq:expibassum} holds. 
The UGASs ensures the existence of $\beta\in\KL$ satisfying 
\eqref{eq:exxestimate} 
for all $\phi_0 \in X$ and all $t\in\R_+$ 
in terms of $X=H^1_0(0,\pi)\times H^1_0(0,\pi)$. 
Interestingly, 
in addition, for system \eqref{GekoppelteNonLinSyst2}, 
Agmon's and Friedrichs' inequalities yield 
\begin{align*}
\left\| \phi(t,\phi_0,0) \right\|_{L_\infty(0,\pi) \times L_\infty(0,\pi)}
\leq  \sqrt{2}\beta(\left\| \phi_0 \right\|_{X},t) .
\end{align*}
for all $\phi_0 \in X$ and all $t\in\R_+$.


\section{Conclusion}\label{sec:conc}

We addressed the problem of stability of interconnected nonlinear parabolic systems.
A small-gain criterion has been proposed together with a method to 
construct Lyapunov functions of interconnected systems.
We emphasized the importance of a correct choice of state spaces 
in accordance of iISS subsystems which are not ISS. 
In ISS literature about parabolic systems \cite{DaM13,MaP11} the systems over $L_p$-spaces (which is the simplest possible case) have been studied most extensively.
However, as indicated in \cite{MiI14b}, the presence of a 
bilinear term in a PDE makes the $L_p$ setting 
break down. 
Indeed, pointwise multiplication of state and input 
variables in PDEs defined on $L_2$ state space 
cannot be bounded by the product of 
the spatial $L_2$-norm of the state and the spatial $L_2$-norm of 
the input, while this is true for norms in Euclidean space in case of ODEs. 
Importantly, when two system are connected to each other, a choice 
of state and input and spaces of one system affects the pair of 
the other systems. 
Thus, the bilinearity makes the choice in Table \ref{table:ISS} 
useless, while the choice is often 
satisfactory for interconnections of ISS subsystems. 
In this paper, tools to construct Lyapunov 
functions characterizing iISS of infinite-dimensional systems have 
been developed,  which 
are not covered by ISS Lyapunov functions. In addition, new methods 
for construction of ISS Lyapunov functions for parabolic systems 
over Sobolev spaces have been proposed as well. 
These new developments allowed one to formulate 
interconnections involving iISS subsystems in the 
setting as in Table \ref{table:iISS}, and they have led 
successfully to a small-gain theorem by which 
stability and robustness can be established 
for a class of nonlinear parabolic systems without 
requiring ISS properties. 

%

\section*{Appendix}
For $L>0$, let $W^{2,1}(0,L)$ denote a Sobolev space
of functions $x\in L_2(0,L)$ which have the first 
order weak derivatives, all of which belong to 
$L_2(0,L)$. 

\begin{proposition}[Young's inequality]
\label{theorem:Young}
For all $a,b \geq 0$ and all $\omega,p>0$ it holds
\begin{eqnarray}
\label{ineq:Young}
ab \leq \frac{\omega}{p}a^p + \frac{1}{\omega^{\frac{1}{p-1}}} \frac{p-1}{p}b^{\frac{p}{p-1}}.
\end{eqnarray}
\end{proposition}

\begin{proof}
See \cite[p. 20]{Mit64}.
\end{proof}

\begin{proposition}[$\Kinf$-inequality]
\label{theorem:Kinf}
For all $a,b \geq 0$, for all $g \in \Kinf$ and all $\omega>0$ it holds
\begin{eqnarray}
\label{ineq:Kinf}
ab \leq \omega a g(a)  + b g^{-1}(\tfrac{b}{\omega}).
\end{eqnarray}
\end{proposition}

\begin{proof}
Follows from estimate of $ab$ for $b \leq \omega g(a)$ and $b \geq \omega g(a)$.
\end{proof}

\begin{proposition}[Jensen's inequality]
\label{theorem:Jensen}
For any convex $f:\R \to \R$ and any summable $x$  
\begin{eqnarray}
\label{ineq:Jensen}
 \int_0^L f(x(l,t)) dl \geq L f\Big( \frac{1}{L} \int_0^L x(l,t)dl\Big).
\end{eqnarray}
\end{proposition}

\begin{proof}
See \cite[p. 705]{Eva2010}.
\end{proof}

\begin{proposition}[Poincare's inequality]
\label{theorem:Wirtinger}
For every $x \in W^{2,1}(0,L)$ it holds that
\begin{align}
\label{Wirtinger_Variation_Ineq}
\frac{4 L^2}{\pi^2} \int_0^L{\left( \frac{\partial x(l)}{\partial l} \right)^2 dl} \geq \int_0^L {x^2(l)}dl
\end{align}
\end{proposition}

\begin{proposition}[Friedrichs' inequality]
\label{theorem:Friedrichs}
For all $f \in H^1_0(0,L) \cap H^2(0,L)$ it holds that
\begin{align}
\label{ineq:Friedrichs}
\frac{L^2}{\pi^2} \int_0^L{\left( \frac{\partial x(l)}{\partial l} \right)^2 dl} \geq \int_0^L {x^2(l)}dl
\end{align}
\end{proposition}


\begin{proposition}[Agmon's inequality]
\label{theorem:Agmon}
For all $f \in H^1(0,L)$ it holds that
\begin{align}
\label{ineq:Agmon}
\|f\|^2_{L_{\infty}(0,L)} \leq |f(0)|^2 + 2\|f\|_{L_{2}(0,L)}\Big\|\frac{df}{dl}\Big\|_{L_{2}(0,L)}.
\end{align}
\end{proposition}

\begin{proof}
See \cite[Lemma 2.4., p. 20]{KrS08}.
\end{proof}

\bibliographystyle{abbrv}
\bibliography{AndersMir}

%

\end{document}